\documentclass[11pt]{amsart}

\usepackage[T1]{fontenc}
\usepackage[utf8]{inputenc}
\usepackage{lmodern}
\usepackage[margin=1.15in]{geometry}
\usepackage{amsmath,amssymb,mathtools}
\usepackage{enumitem}
\usepackage{microtype}
\usepackage{xcolor}
\usepackage[colorlinks=true,linkcolor=blue!45!black,
  citecolor=blue!45!black,urlcolor=blue!45!black]{hyperref}
\hypersetup{
  pdftitle={Unfriendly partitions of locally finite Borel graphs},
  pdfauthor={Jos\'e de Jes\'us Pelayo-G\'omez}}

\newtheorem{theorem}{Theorem}[section]
\newtheorem{maintheorem}{Theorem}

\newtheorem{maincorollary}[maintheorem]{Corollary}
\newtheorem{proposition}[theorem]{Proposition}
\newtheorem{lemma}[theorem]{Lemma}
\newtheorem{corollary}[theorem]{Corollary}
\theoremstyle{definition}
\newtheorem{definition}[theorem]{Definition}
\theoremstyle{remark}
\newtheorem{remark}[theorem]{Remark}

\newcommand{\N}{\mathbb N}
\newcommand{\two}{\{0,1\}}
\newcommand{\Ezero}{E_0}
\newcommand{\Bad}{\operatorname{Bad}}
\newcommand{\sat}[1]{[#1]_{E_G}}
\newcommand{\ind}{\mathbf 1}
\DeclareMathOperator{\Ch}{Ch}
\DeclareMathOperator{\att}{att}
\DeclareMathOperator{\rt}{root}

\title[Unfriendly partitions of Borel graphs]
{Unfriendly partitions of locally finite Borel graphs}
\author{Jos\'e de Jes\'us Pelayo-G\'omez}
\address{Independent researcher}
\email{pelayuss@gmail.com}
\thanks{ORCID: \href{https://orcid.org/0009-0007-2619-6872}
  {0009-0007-2619-6872}.}
\subjclass[2020]{Primary 03E15; Secondary 05C15, 28A05, 54H05}
\keywords{Unfriendly partition, Borel graph, measurable colouring,
  Baire measurable colouring, hyperfinite equivalence relation}
\date{}

\begin{document}

\begin{abstract}
We answer in the negative the question of Thomas, recorded in
\cite{Co14,CMU20,CT21}, of whether every locally finite Borel graph
admits a Borel unfriendly partition.
Our counterexample has unbounded degree and is closed on a
zero-dimensional Polish space; its connectedness relation is hyperfinite,
and its components are bipartite and one-ended.
Every unfriendly colouring is proper. Together with a parity obstruction,
this rigidity rules out Baire measurable colourings that are unfriendly
on a comeager set, and measurable colourings that are unfriendly almost
everywhere for a quasi-invariant probability of finite average degree.
In the positive direction, a Borel graph of maximum degree at most four
admits a Borel unfriendly colouring whenever each component contains a
cycle or a vertex of degree at most two. This reduces the Borel problem
in maximum degree three to cubic forests and, with a theorem of
Conley--Marks--Unger, gives Baire measurable unfriendly colourings for
all Borel graphs of maximum degree at most four.
\end{abstract}

\maketitle

\section{Introduction}

An \emph{unfriendly colouring} of a locally finite graph $H$ is a map
$c\colon V(H)\to\two$ such that every vertex has at least as many neighbours
of the opposite colour as of its own colour. Equivalently,
\begin{equation}\label{eq:unfriendly}
  |\{w\in N_H(v):c(w)\ne c(v)\}|
  \ \ge\ |\{w\in N_H(v):c(w)=c(v)\}|
  \qquad(v\in V(H)).
\end{equation}
The two colour classes form an \emph{unfriendly partition}. For a colouring
which need not satisfy \eqref{eq:unfriendly} everywhere, write $\Bad_H(c)$
for the set of vertices at which the inequality fails.

Every locally finite graph admits an unfriendly colouring
\cite{AMP90}. Indeed, given a finite set of vertices, take a maximum cut
of the finite subgraph induced by those vertices and their neighbours.
It satisfies \eqref{eq:unfriendly} at each of the specified vertices;
compactness then gives a colouring satisfying all the inequalities.
Thomas asked whether every locally finite Borel graph likewise admits
a Borel unfriendly partition \cite{Co14,CMU20,CT21}. The existence
theorems discussed below give reason to expect a positive answer.
We show that local finiteness alone does not suffice.

\begin{maintheorem}\label{cor:borel}
There is a locally finite Borel graph with no Borel unfriendly partition.
\end{maintheorem}

Conley \cite[Remark~2.2 and acknowledgments]{Co14} records the locally
finite question and Thomas's communication of it at MAMLS 2012;
Conley--Tamuz \cite[p.~3 of the arXiv version]{CT21} note that no
locally finite counterexample was known. Kechris and Marks
\cite[Problem~8.7]{KM20} record Weiss's locally countable version,
which already has an elementary category obstruction
(Remark~\ref{rem:flips}). The content of Theorem~\ref{cor:borel}
is local finiteness.

Several positive results impose additional hypotheses on the graph or
the measure. Conley--Tamuz \cite[Theorems~1--3]{CT21} prove existence
almost everywhere for graphs preserving a probability measure of finite
average degree, and for bounded-degree graphs with a quasi-invariant
probability whose edge cocycle is sufficiently close to $1$. They also
obtain Borel unfriendly colourings for bounded-degree graphs of
subexponential growth.
Conley--Marks--Unger \cite[Theorem~1.7]{CMU20} obtain \emph{strongly
unfriendly} colourings, with at most one same-colour neighbour per vertex,
for locally finite acyclic graphs of minimum degree at least two: Baire measurable
ones, and measurable ones under measured hyperfiniteness.

The same example gives both measure and category obstructions, even
when failures on a negligible set are allowed.

\begin{maintheorem}\label{thm:main}
There are a zero-dimensional Polish space $Y$ and a closed, locally finite
graph $G\subseteq Y^2$ with the following properties.
\begin{enumerate}[label=\textup{(\roman*)},leftmargin=*]
\item The connectedness relation $E_G$ is hyperfinite. Every component of
  $G$ is bipartite and has exactly one end. All degrees are odd and at
  least five, and the degrees are unbounded in every component.
\item Every unfriendly colouring of $G$ is a proper two-colouring.
\item No Baire measurable colouring of $G$ is unfriendly on a comeager
  subset of $Y$.
\item There is an atomless, fully supported, $E_G$-quasi-invariant Borel
  probability measure $\lambda$ on $Y$ such that
  \[
    \int_Y \deg_G(y)\,d\lambda(y)<\infty,
  \]
  but no $\lambda$-measurable colouring of $G$ is unfriendly
  $\lambda$-almost everywhere. The relation $E_G$ admits no invariant
  Borel probability measure.
\end{enumerate}
The space $Y$ is an explicitly specified $G_\delta$ subset of
$2^{\N}\times\two\times\N$, with its relative topology.
\end{maintheorem}

Theorem~\ref{cor:borel} follows from Theorem~\ref{thm:main}\textup{(iii)},
since Borel colourings are Baire measurable.
In \textup{(iv)}, a change of density gives integrability; the absence
of an invariant probability distinguishes the example from \cite[Theorem~1]{CT21}.

With the locally finite question answered negatively by
Theorem~\ref{thm:main}, attention turns to its bounded-degree form
\cite[p.~90]{KM20}, \cite[Problem~9.7]{GV25}. Our next result reduces
this remaining Borel question in maximum degree three to cubic forests
(Corollary~\ref{cor:forest-reduction}).

\begin{maintheorem}\label{thm:low-degree}
Let $H$ be a Borel graph of maximum degree at most four on a standard
Borel space. If each connected component of $H$ contains a cycle or a
vertex of degree at most two, then $H$ admits a Borel unfriendly
colouring.
\end{maintheorem}

The degrees in this statement are those of the original graph. There
is no growth assumption and no bound on the distances between cycles
or vertices of degree at most two. In maximum degree three, the only
components left untreated are cubic trees; in maximum degree four,
they are trees with degrees in $\{3,4\}$. Any Borel unfriendly colouring
on these remaining components extends to the whole graph
(Corollary~\ref{cor:forest-reduction}). Combining this reduction with
the Baire measurable forest theorem of Conley--Marks--Unger yields the
following consequence.

\begin{maincorollary}\label{cor:low-degree-baire}
Every Borel graph of maximum degree at most four on a Polish space
admits a Baire measurable unfriendly colouring.
\end{maincorollary}

The counterexample starts with a tree whose proper two-colourings have
a parity obstruction to regularity. Adding leaves and replacing vertices
by finite fibres forces every unfriendly colouring to induce a proper
colouring of that tree. For Theorem~\ref{thm:low-degree}, finite subgraphs
serve as certificates for Borel recolouring: their monochromatic-edge
counts never increase, and each vertex has a certificate whose count
decreases whenever it flips. This proves pointwise stabilization on
the core; attached trees are then coloured by parity.

Sections~\ref{sec:base}--\ref{sec:regularity} construct the counterexample
and prove its regularity obstructions. Section~\ref{sec:low-degree}
proves the positive theorem and its consequences, including a measurable
version under hyperfiniteness. Section~\ref{sec:scope} discusses the
remaining questions and an explicit sequence of approximate solutions.

\subsection*{Conventions}

All graphs below are simple and undirected; as relations, they are
irreflexive and symmetric. For a graph $H$, write $E_H$ for its
connectedness relation and $N_H(v)$ for the neighbourhood of $v$.
We put $\N=\{0,1,2,\ldots\}$. Measurability
with respect to a Borel measure always allows its completion. A map into
$\two$ is \emph{Baire measurable} if its colour classes have the Baire
property. An equivalence relation is \emph{hyperfinite} if it is an
increasing union of Borel equivalence relations with finite classes.
A measure is \emph{quasi-invariant} for a countable Borel equivalence
relation if the saturation of every null Borel set is null. It is
\emph{invariant} if every Borel bijection between Borel subsets whose
graph lies in the relation preserves measure.

\begin{remark}[The locally countable formulation]\label{rem:flips}
On $2^{\N}$, join $x$ to $\sigma_jx$, obtained by flipping coordinate $j$,
for every $j$. Given a Baire measurable colouring $c$, choose a comeager
set $C$ on which $c$ is continuous. On the comeager set
$C\cap\bigcap_j\sigma_j^{-1}(C)$, convergence $\sigma_jx\to x$ implies
$c(\sigma_jx)=c(x)$ for all sufficiently large $j$. Each such vertex has
infinitely many same-colour neighbours and only finitely many
opposite-colour neighbours, so the cardinal inequality defining an
unfriendly colouring fails. This graph has countably infinite degree
at every vertex.
\end{remark}

\section{An explicit tree and its parity obstruction}\label{sec:base}

Let
\[
 X=\{x\in 2^{\N}:x(j)=1\text{ for infinitely many }j\},
\]
with the relative topology of Cantor space and the fair Bernoulli
probability measure $\mu$. The set $X$ is an invariant, conull
$G_\delta$ for the relation $\Ezero$ of eventual equality. For a finite
set $A\subseteq\N$, let $\sigma_A$ flip precisely the coordinates in $A$;
write $\sigma_j=\sigma_{\{j\}}$. These maps are measure-preserving
homeomorphisms of $X$, and $\sigma_jx\to x$ as $j\to\infty$.

For $x\in X$, define
\[
 k(x)=\min\{j:x(j)=1\},\qquad p(x)=\sigma_{k(x)}x.
\]
Thus $p$ changes the first $1$ of $x$ to $0$. Let $T$ be the graph on $X$
whose edges are the pairs $\{x,p(x)\}$.
We exclude finite-support sequences because their class has the
parentless root $0^\infty$, with infinitely many children; on $X$ every
vertex has a parent and only finitely many children.

\begin{lemma}\label{lem:tree}
The graph $T$ is closed and locally finite. Its components are exactly
the $\Ezero$-classes in $X$, and each component is a one-ended tree.
If $x=0^k1y$, then
\[
 \deg_T(x)=k+1,\qquad
 \Ch_T(x)=\{\sigma_jx:j<k\},\qquad
 D_T(x)=\{z1y:z\in 2^k\},
\]
where $\Ch_T(x)=p^{-1}(x)$ and $D_T(x)$ is the set of descendants of $x$,
including $x$ itself.
\end{lemma}

\begin{proof}
The descriptions of children and descendants follow by deleting,
successively, the $1$'s before position $k$. In particular,
$|D_T(x)|=2^k$. The first-$1$ position strictly increases under $p$.
Thus an edge joining vertices with different first-$1$ positions must
be the parent edge of the vertex with the smaller position.
If a cycle existed, a vertex on it with minimal first-$1$ position would
have both cycle neighbours as its parent, which is impossible.

Every edge joins eventually equal sequences. Conversely, if $x$ and $x'$
agree from some position $n$ onwards, deleting all their $1$'s before $n$
gives a common ancestor. Hence the components are exactly the
$\Ezero$-classes. The parent iterates of any vertex form an infinite ray.
Removing a parent edge $\{x,p(x)\}$ leaves the finite set $D_T(x)$ on
the side of $x$. Thus every ray eventually follows the parent direction,
and any two parent rays in a component eventually merge. This proves
one-endedness.

Finally, $k$ is locally constant and $p$ is continuous: near a point with
$k(x)=k$, the map $p$ agrees with the fixed homeomorphism $\sigma_k$.
The graph of $p$ is closed in $X^2$, and so is its inverse. Their union is
$T$.
\end{proof}

\begin{lemma}[Parity obstruction]\label{lem:parity}
The graph $T$ has no $\mu$-measurable proper two-colouring, even on an
$\Ezero$-invariant conull Borel subset of $X$. It has no Baire measurable
proper two-colouring, even on an $\Ezero$-invariant comeager Borel subset
of $X$.
\end{lemma}

\begin{proof}
Each edge of $T$ flips one coordinate. The parity of the length of a path
from $x$ to $x'$ therefore equals the parity of the number of coordinates
on which $x$ and $x'$ differ. Any proper two-colouring $b$ on an invariant
set consequently satisfies
\begin{equation}\label{eq:flip}
 b(\sigma_jx)=1-b(x)\qquad(j\in\N).
\end{equation}

For any $\mu$-measurable map $b\colon X\to\two$, we claim that
\begin{equation}\label{eq:asymptotic}
 \mu\{x:b(\sigma_jx)\ne b(x)\}\longrightarrow0.
\end{equation}
Given $\epsilon>0$, approximate $b$ in measure by a map $g$ depending only
on the first $n$ coordinates, with $\mu(b\ne g)<\epsilon$. Such
approximations exist since the finite-coordinate cylinder sets generate
the Bernoulli sigma-algebra. For $j\ge n$, we have $g\circ\sigma_j=g$,
and measure preservation gives
\[
 \mu(b\circ\sigma_j\ne b)
 \le\mu(b\circ\sigma_j\ne g\circ\sigma_j)+\mu(g\ne b)
 <2\epsilon.
\]
This proves \eqref{eq:asymptotic}, which contradicts \eqref{eq:flip}.
For a colouring on an invariant conull Borel set, extend it by $0$ on
the complement; then \eqref{eq:flip} holds almost everywhere and gives
the same contradiction.

For category, a Baire measurable map $b\colon X\to\two$ is continuous
on a comeager subspace $C$ of $X$. To see this, choose an open set
differing from $b^{-1}(1)$ by a meager set, and remove both that meager
set and the boundary of the open set. The set
\[
 C^*=C\cap\bigcap_{j\in\N}\sigma_j^{-1}(C)
\]
is comeager. For $x\in C^*$, the convergence $\sigma_jx\to x$ and
continuity on $C$ imply $b(\sigma_jx)=b(x)$ for all sufficiently large
$j$. This contradicts \eqref{eq:flip}. A Baire measurable colouring on
an invariant comeager Borel subset extends by $0$ to a Baire measurable
map on $X$; intersecting $C^*$ with that subset proves the stated
relative version.
\end{proof}

\section{Finite replacements forcing proper colourings}\label{sec:construction}

\subsection{Making the base degrees odd}

Put $\epsilon_k=1$ if $k$ is odd and $\epsilon_k=0$ if $k$ is even, and let
\[
 B=\{x\in X:k(x)\text{ is odd}\},\qquad
 Z=(X\times\{0\})\cup(B\times\{1\}).
\]
The set $B$ is clopen in $X$. Define a forest $F$ on $Z$ by putting a copy
of $T$ on $X\times\{0\}$ and adjoining the leaf $(x,1)$ to $(x,0)$ for
every $x\in B$. Give $F$ the parent map
\[
 P(x,0)=(p(x),0),\qquad P(x,1)=(x,0).
\]
The forest is closed, locally finite and one-ended on every component.
Its base vertices have degree $k(x)+1+\epsilon_{k(x)}$, and its leaves
have degree one; thus every degree is odd. Every vertex has finitely
many descendants.

Assign a positive integer multiplicity to each vertex by
\begin{equation}\label{eq:mdef}
 m(x,1)=1,\qquad
 m(x,0)=M_{k(x)},\qquad
 M_k=\frac{3^{k+1}-(-1)^k}{2}.
\end{equation}
Equivalently, $M_0=1$ and
\begin{equation}\label{eq:mrec}
 M_k=1+2\sum_{j<k}M_j+2\epsilon_k.
\end{equation}
Indeed, \eqref{eq:mrec} gives
$M_{k+1}=3M_k+2(-1)^k$, which verifies \eqref{eq:mdef} by induction.
The first values are $1,5,13,41,121$. In particular, every $M_k$ is odd.
Writing $\Ch_F(v)=P^{-1}(v)$, we obtain
\begin{equation}\label{eq:mass}
 m(v)=1+2\sum_{u\in\Ch_F(v)}m(u).
\end{equation}
This includes the case of a leaf. Consequently,
\begin{equation}\label{eq:dominance}
 m(P(v))\ge 1+2m(v)>\sum_{u\in\Ch_F(v)}m(u)
 \qquad(v\in Z).
\end{equation}

\subsection{The replacement graph}

Let
\begin{equation}\label{eq:Y}
 Y=\{(x,t,i):(x,t)\in Z,\ i\in\N,\ i<m(x,t)\}.
\end{equation}
For $v=(x,t)\in Z$, write
$Y_v=\{(x,t,i):i<m(v)\}$ for its replacement fibre. Let
$r\colon Y\to Z$ forget the last coordinate. Define $G$ by
\begin{equation}\label{eq:G}
 y\mathrel G y'\quad\Longleftrightarrow\quad
 r(y)\mathrel F r(y').
\end{equation}
Thus each $F$-edge is replaced by the complete bipartite graph between
its two fibres, and there are no edges within a fibre. Every vertex in
$Y_v$ has the same neighbourhood and the same degree
\begin{equation}\label{eq:degree}
 a(v)=\sum_{w\in N_F(v)}m(w).
\end{equation}
This is a finite odd integer, since both the number of summands and each
summand are odd.

\begin{proposition}[Rigidity]\label{prop:rigidity}
The unfriendly colourings of $G$ are exactly the maps $c=b\circ r$, where
$b\colon Z\to\two$ is a proper two-colouring of $F$. In particular,
every unfriendly colouring of $G$ is proper. The same assertion holds
on any union of connected components.
\end{proposition}

\begin{proof}
Suppose $c$ is unfriendly. All vertices in a fibre $Y_v$ have identical
neighbourhoods of odd size. Such a neighbourhood has a unique majority
colour. The inequality \eqref{eq:unfriendly} forces every vertex of
$Y_v$ to have the other colour, so the fibre is monochromatic. Let $b(v)$
be its colour.

The neighbourhood of a vertex in $Y_v$ consists of the parent fibre
$Y_{P(v)}$ and all the child fibres. By \eqref{eq:dominance}, the parent
fibre contains strictly more vertices than all the child fibres
together. Its colour is therefore the strict majority colour in that
neighbourhood. Hence $b(v)\ne b(P(v))$ for every $v$, and $b$ is proper
on $F$. Conversely, the lift of a proper two-colouring of $F$ is proper
on $G$, and is thus unfriendly. The proof applies unchanged to a union
of components.
\end{proof}

\begin{proposition}\label{prop:structure}
The space $Y$ and the graph $G$ have all the topological and combinatorial
properties in Theorem~\ref{thm:main}\textup{(i)}. Moreover, the degree
function is locally constant, and the map
\[
 s\colon X\to Y,\qquad s(x)=(x,0,0),
\]
is a homeomorphism onto a clopen subset of $Y$.
\end{proposition}

\begin{proof}
The function $m$ is locally constant on $Z$. Thus $Y$ is clopen in
$X\times\two\times\N$, and hence is a $G_\delta$ subset of
$2^{\N}\times\two\times\N$. It is zero-dimensional and Polish. The
map $r$ is continuous and $F$ is closed, so \eqref{eq:G} shows that $G$
is closed. The assertion about $s$ follows from $m(x,0)\ge1$.

Local finiteness and odd degrees follow from \eqref{eq:degree}.
If the first two $1$'s of $x$ occur at positions $k<\ell$, then
\begin{equation}\label{eq:explicit-degree}
 a(x,0)=M_\ell+\frac{M_k-1}{2},\qquad
 a(x,1)=M_k\quad(x\in B).
\end{equation}
These formulas also prove local constancy. Since $\ell\ge1$, the first
degree is at least $M_1=5$; the second is at least five because $k$ is
odd. Along the successive parent iterates of any base vertex, the
first-$1$ positions tend to infinity. Formula \eqref{eq:explicit-degree}
then shows that degrees are unbounded in every component.

The bipartition of each tree component of $F$ lifts to $G$. To check
one-endedness, enlarge any finite deletion in a $G$-component to whole
fibres. Deleting their projections from $F$ leaves one infinite component
and finitely many finite ones, since $F$ is locally finite and one-ended.
The infinite component lifts to a connected graph: each fibre is
connected through an adjacent fibre. Only finitely many vertices remain
outside it. Restoring the enlarged deletion outside the original set
still leaves exactly one infinite component, proving one-endedness.

Finally, let $\pi\colon Y\to X$ be $\pi(x,t,i)=x$. Since $F$ has no
isolated vertices, all the vertices over a fixed $x$ belong to one
$G$-component, and
\begin{equation}\label{eq:connectedness}
 y\mathrel{E_G}y'\quad\Longleftrightarrow\quad
 \pi(y)\mathrel{\Ezero}\pi(y').
\end{equation}
Define $R_n$ on $Y$ by
\[
 y\mathrel{R_n}y'\quad\Longleftrightarrow\quad
 \pi(y)(j)=\pi(y')(j)\text{ for every }j\ge n.
\]
Each $R_n$ is a Borel equivalence relation. A class involves at most
$2^n$ base sequences, each supporting finitely many vertices, so it is
finite. The relations increase and their union is $E_G$.
\end{proof}

\section{Category and measure}\label{sec:regularity}

We first describe the partial homeomorphisms which allow negligible
exceptional sets to be removed componentwise. For $t\in\two$ and
$i\in\N$, put
\[
 U_{t,i}=\{x\in X:(x,t,i)\in Y\},\qquad
 S_{t,i}=\{(x,t,i):x\in U_{t,i}\}.
\]
These are clopen in $X$ and $Y$, respectively. For a finite
$A\subseteq\N$ and indices $(t,i),(u,j)$, consider the map
\begin{equation}\label{eq:partial-homeo}
 (x,t,i)\longmapsto(\sigma_Ax,u,j)
 \quad\text{on }x\in U_{t,i}\cap\sigma_A^{-1}(U_{u,j}).
\end{equation}
It is a homeomorphism between clopen subsets of $Y$. There are countably
many such maps, and their graphs cover $E_G$ by
\eqref{eq:connectedness}. In particular, saturations of Borel sets are
Borel, and saturations of meager sets are meager.

\subsection{The Baire obstruction}

\begin{proposition}\label{prop:baire}
No Baire measurable map $c\colon Y\to\two$ is unfriendly on a comeager
set.
\end{proposition}

\begin{proof}
Suppose otherwise. Choose a meager Borel set $N$ containing all vertices
at which $c$ fails to be unfriendly, and put
$Y_0=Y\setminus\sat N$. The preceding partial homeomorphisms show that
$Y_0$ is an invariant comeager Borel set. The restriction of $c$ is
unfriendly at every vertex of $G\upharpoonright Y_0$, since invariance
ensures that all its neighbours remain in $Y_0$.

Set $X_0=s^{-1}(Y_0)$. The image of $s$ is clopen, so $X_0$ is comeager
in $X$, and $b(x)=c(s(x))$ is Baire measurable on $X_0$. Equation
\eqref{eq:connectedness} makes $X_0$ $\Ezero$-invariant. By
Proposition~\ref{prop:rigidity}, $b$ is a proper two-colouring of
$T\upharpoonright X_0$, contradicting Lemma~\ref{lem:parity}.
\end{proof}

\subsection{A quasi-invariant probability with integrable degrees}

Define a Borel measure on $Y$ by counting the vertices over each base
point:
\begin{equation}\label{eq:eta}
 \eta(A)=\int_X
   \sum_{\substack{t\in\two\colon (x,t)\in Z}}
   \sum_{i<m(x,t)}\ind_A(x,t,i)\,d\mu(x).
\end{equation}
Its restriction to each sheet $S_{t,i}$ is a copy of
$\mu\upharpoonright U_{t,i}$, so it is sigma-finite, atomless and fully
supported. For Borel $A\subseteq Y$,
\begin{equation}\label{eq:null-projection}
 \eta(A)=0\quad\Longleftrightarrow\quad\mu(\pi(A))=0.
\end{equation}
Here $\pi(A)$ is Borel, since it is the countable union of the
projections of $A\cap S_{t,i}$ along homeomorphisms.

\begin{lemma}\label{lem:invariant}
The measure $\eta$ is $E_G$-invariant. In particular, the saturation of
every $\eta$-null Borel set is $\eta$-null.
\end{lemma}

\begin{proof}
Each map in \eqref{eq:partial-homeo} preserves $\eta$, because it acts
between sheets by a finite coordinate flip and preserves $\mu$.
Any Borel partial bijection with graph in $E_G$ can be partitioned into
countably many Borel restrictions of these maps: enumerate the maps
and assign a point to the first one which agrees with the bijection
there. The images of those pieces are disjoint. Countable additivity
proves invariance. The null-saturation assertion also follows directly
by taking the countable union of the images in
\eqref{eq:partial-homeo}.
\end{proof}

For an equivalent probability of finite average degree, give each fibre
total weight inversely proportional to one plus its degree. Let
\begin{equation}\label{eq:normalization}
 C=\int_X\sum_{\substack{t\in\two\colon (x,t)\in Z}}
       \frac{1}{1+a(x,t)}\,d\mu(x).
\end{equation}
Then $0<C\le2$. Define
\begin{equation}\label{eq:lambda}
 \lambda(A)=\frac1C\int_X
   \sum_{\substack{t\in\two\colon (x,t)\in Z}}
   \frac{1}{m(x,t)(1+a(x,t))}
   \sum_{i<m(x,t)}\ind_A(x,t,i)\,d\mu(x).
\end{equation}

\begin{proposition}\label{prop:measure}
The measure $\lambda$ is an atomless, fully supported,
$E_G$-quasi-invariant Borel probability measure of finite average degree. No
$\lambda$-measurable colouring of $G$ is unfriendly almost everywhere.
The same nonexistence assertion holds for every Borel probability
measure equivalent to $\lambda$.
\end{proposition}

\begin{proof}
Summing the density in \eqref{eq:lambda} over a fibre cancels its
multiplicity, so \eqref{eq:normalization} gives $\lambda(Y)=1$.
The density with respect to $\eta$ is everywhere positive and finite.
Thus the two measures have the same null sets. Atomlessness and full
support follow from those of $\eta$, and quasi-invariance follows
from Lemma~\ref{lem:invariant}. Moreover,
\[
 \int_Y \deg_G(y)\,d\lambda(y)
 =\frac1C\int_X\sum_{\substack{t\in\two\colon (x,t)\in Z}}
       \frac{a(x,t)}{1+a(x,t)}\,d\mu(x)
 \le \frac{2}{C}<\infty.
\]

Suppose a $\lambda$-measurable colouring $c$ is unfriendly almost
everywhere. Choose a null Borel set $N$ containing all its failures.
Quasi-invariance makes $Y_0=Y\setminus\sat N$ an invariant conull
Borel set. Put $X_0=s^{-1}(Y_0)$. By equivalence of $\eta$ and
$\lambda$ and \eqref{eq:null-projection}, this is an invariant
$\mu$-conull Borel subset of $X$.

The restriction of $\lambda$ to the clopen sheet $s(X)$ pulls back
under $s$ to a measure with strictly positive density
\[
 \frac{1}{C M_{k(x)}(1+a(x,0))}
\]
with respect to $\mu$. Consequently, $b(x)=c(s(x))$ is
$\mu$-measurable on $X_0$, including for completed measures.
Proposition~\ref{prop:rigidity} makes it a proper two-colouring of
$T\upharpoonright X_0$, contradicting Lemma~\ref{lem:parity}.
Replacing $\lambda$ by an equivalent probability changes neither the
completed measurable sets nor the null sets, and hence does not change
this argument.
\end{proof}

Finite average degree is obtained by changing density within the same
measure class; the obstruction itself depends only on that class.

\subsection{Invariant probabilities and the edge cocycle}

The invariant-probability hypothesis of \cite[Theorem~1]{CT21} cannot
hold for this connectedness relation, under any choice of probability.

\begin{proposition}\label{prop:no-invariant-probability}
The measure $\eta$ is infinite, and $E_G$ admits no invariant Borel
probability measure.
\end{proposition}

\begin{proof}
Since $\mu(k(x)=k)=2^{-k-1}$, we have
\[
 \eta(Y)\ge\int_X M_{k(x)}\,d\mu(x)
       =\sum_{k=0}^{\infty}2^{-k-1}M_k=\infty
\]
by \eqref{eq:mdef}.

Now suppose that $\nu$ is an invariant Borel probability measure for
$E_G$, and define the finite Borel measure $\alpha$ on $X$ by
$\alpha(A)=\nu(s(A))$. For each sheet, the map
$(x,t,i)\mapsto s(x)$ is a partial bijection in $E_G$. Invariance gives
\begin{equation}\label{eq:nu-sheet}
 \nu(\{(x,t,i):x\in A\})=\alpha(A)
 \qquad(A\subseteq U_{t,i}\text{ Borel}).
\end{equation}
If $\alpha(X)=0$, all the countably many sheets have $\nu$-measure
zero, a contradiction. Write $d=\alpha(X)>0$.

The bijections $s(x)\mapsto s(\sigma_Ax)$ show that $\alpha$ is
invariant under every finite coordinate flip. All length-$n$ cylinders
in $X$ therefore have $\alpha$-measure $d2^{-n}$. Since these
cylinders generate the Borel sets of $X$, it follows that
$\alpha=d\mu$. Summing \eqref{eq:nu-sheet} over all sheets yields
$\nu(Y)=d\eta(Y)=\infty$, again a contradiction.
\end{proof}

Propositions~\ref{prop:rigidity}, \ref{prop:structure},
\ref{prop:baire}, \ref{prop:measure}, and
\ref{prop:no-invariant-probability} prove Theorem~\ref{thm:main}.

Conley--Tamuz \cite[Theorem~2]{CT21} also obtain a Borel unfriendly
colouring on an invariant conull set for quasi-invariant probabilities
when degrees are bounded by $d$ and the edge cocycle lies in
$[1-1/d,1+1/d]$. Our example has neither bound. To check the latter,
put $h(y)=1/[m(r(y))(1+a(r(y)))]$, so $d\lambda=C^{-1}h\,d\eta$.
Invariance of $\eta$ gives the cocycle
$\rho(y,y')=h(y')/h(y)$, with the convention
$\lambda(f(A))=\int_A\rho(y,f(y))\,d\lambda(y)$ for partial bijections
in $E_G$. If the first two $1$'s of $x$ occur at $1<\ell$, then
\[
 \rho\bigl((x,0,0),(x,1,0)\bigr)=\frac{5(M_\ell+3)}{6}.
\]
These are edges of $G$, and each such cylinder has positive measure.
The cocycle is therefore essentially unbounded on edges.

\section{A positive theorem in maximum degree four}\label{sec:low-degree}

We now prove Theorem~\ref{thm:low-degree}. Throughout this section,
$H$ is a Borel graph of maximum degree at most four on a standard
Borel space $W$. Write
\[
 d(v)=\deg_H(v),\qquad q(v)=\lfloor d(v)/2\rfloor,
 \qquad L=\{v\in W:d(v)\le2\}.
\]
For a subgraph $J\subseteq H$ and a colouring $c$, let $s_c^J(v)$
and $o_c^J(v)$ denote the numbers of neighbours of $v$ in $J$ with,
respectively, the same colour and the opposite colour. The target
inequality on $H$ is $s_c^H(v)\le q(v)$.

\subsection{The core of finite certificates}

\begin{definition}
An \emph{admissible certificate} is a finite connected subgraph
$Q\subseteq H$ with at least one edge such that
\begin{equation}\label{eq:admissible-leaf}
 \deg_Q(v)=1\quad\Longrightarrow\quad d(v)\le2.
\end{equation}
The subgraph need not be induced. Define
\[
 K=\bigcup\{V(Q):Q\text{ is an admissible certificate}\},
 \qquad J=H\upharpoonright K.
\]
\end{definition}

Every cycle is a certificate, as is every nontrivial simple path
with both endpoints in $L$. This core differs from the one obtained
by iteratively deleting leaves: a tree of minimum degree three has
no admissible certificate.

\begin{lemma}\label{lem:certificate-core}
The set $K$ is Borel. In every component of $H$ that contains a
cycle or at least two vertices of $L$, the following hold:
\begin{enumerate}[label=\textup{(\roman*)},leftmargin=*]
\item The intersection with $K$ is nonempty and connected, and
  contains all vertices of $L$ in the component.
\item Every simple path between vertices of $K$ lies in $K$.
\item Every edge of $J$ in the component belongs to an admissible
  certificate contained in $J$.
\item Deleting $K$ leaves a forest, each of whose components is
  joined to $K$ by exactly one edge.
\end{enumerate}
\end{lemma}

\begin{proof}
Membership of $v$ in $K$ is witnessed by a finite subgraph inside
some ball $B_H(v,n)$. Lusin--Novikov gives Borel enumerations of
these finite balls and their finite subgraphs; see \cite{KM20}.
The certificate conditions are Borel, so taking the union over
radii and witnesses shows that $K$ is Borel.

If the component contains a cycle, that cycle is a certificate.
A point of $L$ outside the cycle can be joined to it by a simple path
stopped at its first meeting with the cycle; the union is another
certificate. If the component contains two distinct points of $L$,
a path between them is a certificate, and any other point of $L$
in the component can likewise be joined to one of them.

Suppose $u,v\in K$ lie in the same component. Choose certificates
$Q_u,Q_v$ containing them and any simple path $P$ from $u$ to $v$. The union
$Q_u\cup P\cup Q_v$ is finite and connected. Every new interior
vertex of $P$ has degree at least two in this union. Any leaf of the
union was already a leaf of one of the certificates, and hence lies
in $L$. The union is therefore admissible. This proves (ii) and the
connectedness in (i). Taking $P$ to be an edge proves (iii).

There can be no cycle outside $K$. Each exterior component has an
edge to $K$, since the original component is connected and meets $K$. Two
such edges with distinct endpoints in $K$ would give a simple path
between those endpoints through the exterior tree, contrary to (ii).
Two with the same endpoint in $K$ would give a cycle through the
exterior tree. Both are impossible, proving (iv).
\end{proof}

\subsection{Stabilization from monochromatic-edge counts}

We use local-improvement dynamics as in Conley--Tamuz \cite{CT21}.
The convergence argument uses a family of finite counters. We will
show that every admissible certificate has a nonincreasing number of
monochromatic edges, and that, for each vertex, one certificate records
a strict decrease at every flip of that vertex. These two properties
give pointwise stabilization without a summable potential on a component.

On the induced core $J$ we seek the inequality
\begin{equation}\label{eq:core-quota}
 s_c^J(v)\le q(v)\qquad(v\in K),
\end{equation}
where the quota still uses the original degree in $H$. All edges
outside the core will subsequently be made bichromatic.

Choose a countable Borel proper colouring $\gamma:K\to\N$ of $J$,
which exists for every locally finite Borel graph \cite{KM20}, and
a sequence $(i_n)_{n\in\N}$ in which every integer occurs infinitely
often. Start with $c_0=0$ on $K$. At stage $n$, flip precisely the
vertices in the Borel set
\[
 B_n=\{v\in K:\gamma(v)=i_n,\ s_{c_n}^J(v)>q(v)\}.
\]
This defines a sequence of Borel colourings $(c_n)$; each set $B_n$
is independent in $J$.

\begin{lemma}\label{lem:certificate-stabilization}
Every vertex of $K$ changes colour only finitely often in this
process. The pointwise limit is Borel and satisfies
\eqref{eq:core-quota}.
\end{lemma}

\begin{proof}
For an admissible certificate $Q$, let $M_Q(c)$ be its number of
monochromatic edges. Flipping a single vertex $v\in V(Q)$ changes
this count by
\begin{equation}\label{eq:certificate-drop}
 M_Q(c)-M_Q(c^v)=\deg_Q(v)-2o_c^Q(v).
\end{equation}
Here $c^v$ denotes the colouring obtained by flipping $v$.

Suppose that $v$ is eligible for a flip. If $d(v)\le2$, then
$s_c^J(v)>\lfloor d(v)/2\rfloor$ forces $o_c^J(v)=0$:
for $d(v)=1$ the one available edge must be monochromatic, and for
$d(v)=2$ both edges must be present in $J$ and monochromatic.
Degree zero cannot occur in a certificate. Thus
\eqref{eq:certificate-drop} is positive in this case, including
when $v$ is a leaf of $Q$.

If $d(v)\in\{3,4\}$, admissibility gives $\deg_Q(v)\ge2$, while
\[
 o_c^J(v)\le\deg_J(v)-q(v)-1
             \le d(v)-\lfloor d(v)/2\rfloor-1=1.
\]
Hence \eqref{eq:certificate-drop} is nonnegative. Flips outside
$V(Q)$ do not change $M_Q$. Since the simultaneous flip centres
are independent, no edge is incident to two of them, so their
changes add. It follows that $M_Q(c_n)$ is a nonincreasing sequence
of nonnegative integers for every certificate $Q$.

Fix $v\in K$. Every component meeting $K$ contains a cycle or at
least two points of $L$: an admissible certificate is either cyclic
or a finite nontrivial tree with at least two leaves in $L$.
By Lemma~\ref{lem:certificate-core}, each edge of the finite
$J$-star of $v$ belongs to a certificate. The union of these finitely
many certificates is again connected and admissible, and contains
the whole star. Denote this union by $Q_v$. Whenever $v$ flips,
its contribution to the decrease of $M_{Q_v}$ is
\begin{equation}\label{eq:strict-certificate-drop}
 2s_c^J(v)-\deg_J(v)
 \ge 2\lfloor d(v)/2\rfloor+2-\deg_J(v)\ge1.
\end{equation}
All other contributions are nonnegative. Consequently, $v$ flips
at most $M_{Q_v}(c_0)=|E(Q_v)|$ times. The certificate $Q_v$ is
only a witness for convergence; the process never chooses such
witnesses Borelly.

The pointwise limit $c_K$ is Borel. If
\eqref{eq:core-quota} failed at $v$, then after every vertex of its
finite closed $J$-neighbourhood had stabilized, the next visit to
the class $\gamma^{-1}(\gamma(v))$ would flip $v$ again. This is a
contradiction.
\end{proof}

\subsection{Extension to the attached trees}

\begin{proof}[Proof of Theorem~\ref{thm:low-degree}]
First consider the components containing a cycle or at least two
vertices of $L$. These components form a Borel invariant set:
each of the defining properties has a finite witness in a ball.
Use Lemma~\ref{lem:certificate-stabilization} to colour their core.
For each exterior vertex $x$, Lemma~\ref{lem:certificate-core}
gives a unique closest vertex $\att(x)\in K$, namely the attachment
point of its exterior tree. Both $\att(x)$ and the finite distance
$\ell(x)=d_H(x,K)$ are Borel; for $\att$ one can again use
Lusin--Novikov, with singleton sections. Define
\[
 c(x)=
 \begin{cases}
 c_K(x),&x\in K,\\
 c_K(\att(x))+\ell(x)\pmod2,&x\notin K.
 \end{cases}
\]
Every edge with at least one endpoint outside $K$ is bichromatic.
Thus $s_c^H(v)=s_{c_K}^J(v)\le q(v)$ in $K$, and
$s_c^H(v)=0$ outside $K$.

Under the hypothesis of the theorem, the remaining components
are trees containing exactly one vertex of $L$. They too form a
Borel invariant set. The unique vertex $\rt(x)\in L$ in the component
of $x$ depends Borelly on $x$. Colouring $x$ by
$d_H(x,\rt(x))\pmod2$ is proper and Borel, including on isolated
vertices. Combining the two invariant parts completes the proof.
\end{proof}

\begin{corollary}\label{cor:forest-reduction}
For a Borel graph $H$ of maximum degree at most four, let
\[
 R=\{x\in W:[x]_{E_H}\text{ is a tree all of whose vertices
                         have degree at least three}\}.
\]
Then $R$ is Borel and invariant, and every Borel unfriendly
colouring of $H\upharpoonright R$ extends to a Borel unfriendly
colouring of $H$. In particular:
\begin{enumerate}[label=\textup{(\roman*)},leftmargin=*]
\item Every Borel graph of maximum degree at most three has a Borel
  unfriendly colouring if and only if every cubic Borel forest does.
\item Every Borel graph of maximum degree at most four has a Borel
  unfriendly colouring if and only if every Borel forest with
  degrees in $\{3,4\}$ does.
\end{enumerate}
\end{corollary}

\begin{proof}
The complement of $R$ is the invariant Borel set of components
containing a cycle or a vertex of degree at most two.
Theorem~\ref{thm:low-degree} supplies a Borel unfriendly colouring
there. No edge joins the two invariant parts, so this colouring
can be combined with any prescribed solution on $R$. The two
equivalences follow by inspecting the degrees on $R$; the reverse
implications follow by inclusion of classes.
\end{proof}

\subsection{Baire and measurable consequences}

\begin{proof}[Proof of Corollary~\ref{cor:low-degree-baire}]
Let $W$ carry the given Polish topology and let $R$ be as in
Corollary~\ref{cor:forest-reduction}. Theorem~\ref{thm:low-degree}
colours $W\setminus R$ Borelly. To apply Conley--Marks--Unger
\cite[Theorem~1.7]{CMU20} on the remaining forest without changing
the topology, define an auxiliary graph $\widehat H$ on the Polish
space $W\times\mathbb Z$ as follows.

Give $\mathbb Z$ the discrete topology and fix a bijection
$b:\mathbb Z\setminus\{0\}\to\mathbb Z$.
On $R\times\{0\}$ use a copy of $H\upharpoonright R$.
For each $x\in R$, put a bi-infinite path on
$\{x\}\times(\mathbb Z\setminus\{0\})$, joining $(x,i)$ to
$(x,j)$ when $|b(i)-b(j)|=1$.
For each $x\notin R$, put the usual bi-infinite path on
$\{x\}\times\mathbb Z$. Add no other edges.
The resulting graph is Borel and acyclic, with every degree at
least two and at most four.

The cited theorem gives a Baire measurable strongly unfriendly
colouring of $\widehat H$. Its restriction to the clopen slice
$W\times\{0\}$ gives a Baire measurable map $f:W\to\two$. On $R$ it is strongly
unfriendly for $H$, hence unfriendly since all degrees there are
at least three. On the Borel set $W\setminus R$, replace $f$
by the Borel solution already constructed. The resulting map is
Baire measurable in the original topology and is unfriendly
everywhere.
\end{proof}

The measurable part of the same theorem yields the following consequence.
Recall that $H$ is \emph{$\nu$-hyperfinite} if its connectedness relation
is hyperfinite on an invariant conull Borel set.

\begin{corollary}\label{cor:low-degree-measure}
Let $H$ be a Borel graph of maximum degree at most four on a standard
Borel space $W$, and let $\nu$ be a Borel probability measure on $W$.
If $H$ is $\nu$-hyperfinite, then it admits a $\nu$-measurable unfriendly
colouring.
\end{corollary}

\begin{proof}
Let $R$ be the invariant Borel set in
Corollary~\ref{cor:forest-reduction}. On $W\setminus R$, use the Borel
colouring supplied by Theorem~\ref{thm:low-degree}. If $\nu(R)>0$, equip
$R$ with a compatible Polish topology and normalize
$\nu\upharpoonright R$ to a probability. The forest
$H\upharpoonright R$ remains hyperfinite on an invariant conull Borel
set and has minimum degree at least three. Conley--Marks--Unger
\cite[Theorem~1.7]{CMU20} therefore give a measurable strongly unfriendly
colouring there, which is unfriendly. If $\nu(R)=0$, choose an arbitrary
unfriendly colouring on $R$; it is measurable for the completed measure.
Combining the two invariant parts gives the required colouring on $W$.
\end{proof}

\section{Scope and approximation}\label{sec:scope}

Unfriendly colouring is also studied in distributed computing under
the name \emph{locally optimal cut}; Balliu et al.\ \cite{BBD+25} give
polylogarithmic-round algorithms for bounded-degree finite graphs.
For connections between distributed algorithms and Borel or measurable
combinatorics, see Bernshteyn \cite{Be23} and, for regular trees,
Brandt et al.\ \cite{BCG+22}.

\subsection{The role of the degree bound}

Theorem~\ref{thm:main} shows that hyperfiniteness alone does not extend
Corollary~\ref{cor:low-degree-measure} to all locally finite graphs,
and that bipartiteness cannot replace acyclicity in
\cite[Theorem~1.7]{CMU20}. The replacement graph contains $4$-cycles
between adjacent fibres of size at least two.

The bound four ensures that a vertex eligible to flip has at most one
opposite-colour neighbour in the core. In degree five, it may have two;
a certificate using only those two incident edges would gain
monochromatic edges after the flip. The argument therefore does not
extend directly, but establishes no sharp degree threshold.

Conversely, the counterexample needs unbounded fibre sizes by
\eqref{eq:mass}. Attaching leaves cannot replace these weights:
each leaf opposes its attachment vertex and thus relaxes the
constraint on its other neighbours.

The bounded-degree Borel and measurable questions remain open
\cite[Problem~9.7]{GV25}, with the Borel forest reduction given by
Corollary~\ref{cor:forest-reduction}. It also remains open whether
finite average degree is necessary in \cite[Theorem~1]{CT21}:
our obstruction admits no invariant probability.

\subsection{Approximate solutions}

\begin{remark}[Approximation without a solution]\label{rem:approximation}
The continuous colourings
$c_n(x,t,i)=(t+\sum_{j<n}x(j))\bmod2$ satisfy
\[
 \Bad_G(c_n)=\{(x,0,i)\in Y:k(x)\ge n\}.
\]
Indeed, leaf edges are always proper. If $k(x)<n$, every incident base
edge flips a coordinate below $n$, so all incident edges are proper.
If $k(x)\ge n$, the parent fibre has the same colour and supplies a
strict majority by \eqref{eq:dominance}. These clopen failure sets
decrease to the empty set, and every vertex eventually has all its
incident edges properly coloured. Hence
$\theta(\Bad_G(c_n))\to0$ for every Borel probability $\theta$ on $Y$.
For $\lambda$, the infimum of the failure measures of Borel colourings
is zero, but no Borel colouring attains it. There is no pointwise limit
of $(c_n)$: each successive $1$ in a sequence changes its colour.
\end{remark}

\section*{Use of AI tools}

The author used AI tools from Anthropic (Claude) and OpenAI (Codex) to check computations,
locate and verify references, conduct exploratory finite experiments, and
write the final version of this text. The proofs do not rely on these exploratory
computations. The author independently verified all AI-generated output and all references,
made the final mathematical and editorial decisions, and takes
full responsibility for the accuracy, integrity, and originality of this article.

\end{document}